\documentclass[11pt, a4paper]{article}
\usepackage{hyperref}
\usepackage{enumerate, longtable}
\usepackage{amsmath, amscd, amsfonts, amsthm, amssymb, latexsym, comment, stmaryrd, graphicx}
\usepackage{xcolor}
\usepackage[all]{xy}
\usepackage{a4wide}

\newtheorem{thm}{Theorem}[section]
\newtheorem{lem}[thm]{Lemma}
\newtheorem{prop}[thm]{Proposition}
\newtheorem{cor}[thm]{Corollary}

\theoremstyle{definition}

\theoremstyle{remark}
\newtheorem{rem}[thm]{Remark}
\newtheorem{ex}[thm]{Example}

\newcommand{\FF}{\mathbb{F}}
\newcommand{\NN}{\mathbb{N}}

\newcommand{\ZZ}{\mathbb{Z}}

\newcommand{\GL}{\mathrm{GL}}

\DeclareMathOperator{\Gal}{Gal}
\DeclareMathOperator{\Ind}{Ind}
\newcommand{\im}{\mathrm{im}}

\renewcommand{\check}{\tilde}
\renewcommand{\lhd}{\unlhd}
\renewcommand{\rhd}{\unrhd}

\begin{document}

\title{Hilbertian fields and Galois representations}

\author{Lior Bary-Soroker
\and Arno Fehm
\and Gabor Wiese
}

\maketitle

\begin{abstract}
We prove a new Hilbertianity criterion for fields in towers
whose steps are Galois with Galois group either abelian or a product of finite simple groups.
We then apply this criterion to fields arising from Galois representations. 
In particular we settle a conjecture of Jarden on abelian varieties. 
\end{abstract}

\section{Introduction}

A field $K$ is called {\bf Hilbertian} if for every irreducible polynomial $f(t,X) \in K[t,X]$ that is separable in $X$ there exists $\tau \in K$ such that $f(\tau,X) \in K[X]$ is irreducible. 
This notion stems from Hilbert's irreducibility theorem, 
which, in modern terminology, asserts that number fields are Hilbertian.
Hilbert's irreducibility theorem and Hilbertian fields play an important role in algebra and number theory, in particular in Galois theory and arithmetic geometry, 
cf.\ \cite{FriedJarden}, \cite{Lang}, \cite{MalleMatzat},  \cite{Serre}, \cite{SerreMordell},  \cite{Volklein}.

In the light of this, an important question is under what conditions an extension of a  Hilbertian field is Hilbertian. 
As central examples we mention Kuyk's theorem \cite[Theorem 16.11.3]{FriedJarden}, which asserts that an abelian extension of a Hilbertian field is Hilbertian,
and Haran's diamond theorem \cite{HaranDiamond}, which is the most advanced result in this area. 

In this work we introduce an invariant of profinite groups that we call {\bf abelian-simple length} (see Section \ref{sec:gds} for definition and basic properties)
and prove a Hilbertianity criterion for extensions whose Galois groups have finite abelian-simple length:

\begin{thm}\label{thm:criteria}
Let $L$ be an algebraic extension of a Hilbertian field $K$. Assume there exists a tower of field extensions
\[
K=K_0 \subseteq K_1 \subseteq \cdots \subseteq K_m
\]
such that $L\subseteq K_m$ and
for each $i$ the extension $K_{i}/K_{i-1}$ is Galois with group either abelian or a (possibly infinite) product of  finite simple groups.
Then $L$ is Hilbertian.
\end{thm}

In the special case where $m=1$ and $K_1/K_0$ is abelian we recover Kuyk's theorem.
Already the case when $m=2$ and both $K_1/K_0$ and $K_2/K_1$ are abelian is new. 
As a first application of this criterion we then prove Hilbertianity of certain extensions arising from Galois representations:
  
\begin{thm}\label{thm:main}
Let $K$ be a Hilbertian field, let $L/K$ be an algebraic extension, let $n$ be a fixed integer, and for each prime number $\ell$ let 
$$
 \rho_\ell \colon \Gal(K) \to \GL_n(\overline{\mathbb{Q}}_\ell)
$$ 
be a Galois representation.  Assume that $L$ is fixed by $\bigcap_\ell\ker\rho_\ell$. Then $L$ is Hilbertian.
\end{thm}

This theorem is a strengthening of a recent result of Thornhill,
who in \cite{Thornhill} proves Theorem~\ref{thm:main} under the extra assumptions that $L$ is Galois over $K$ and the image of $\rho_\ell$ is in $\GL_n(\ZZ_\ell)$. 

Both Thornhill's and our research was motivated by a conjecture of Jarden. Namely, when applying Theorem~\ref{thm:main} to the family of Galois representations $\rho_\ell \colon \Gal(K)\to \GL_{2\dim(A)}(\mathbb{Z}_\ell)$ coming from the action of the absolute Galois group $\Gal(K)$ on the Tate module $T_\ell(A)$ of an abelian variety $A$ over $K$, 
Theorem~\ref{thm:main} immediately implies the following result.

\begin{thm}[Jarden's conjecture]\label{conj:Jarden}
Let $K$ be a Hilbertian field and $A$ an abelian variety over $K$. Then every field $L$ between $K$ and $K(A_{tor})$,
the field generated by all torsion points of $A$, is Hilbertian.
\end{thm}

This conjecture was proven by Jarden in \cite{Jarden} for number fields $K$, and in \cite{FehmJardenPetersen} by Jarden, Petersen, and the second author for function fields $K$. Thornhill is able to deduce the special case where $L$ is Galois over $K$, like above.

Jarden's proof in \cite{Jarden} uses Serre's open image theorem,
while Thornhill's proof in \cite{Thornhill} is based on a theorem of Larsen and Pink \cite{LarsenPink} on subgroups of $\GL_n$ over finite fields.
Another key ingredient in both proofs is Haran's diamond theorem. 
The fact that $L/K$ is Galois is crucial in Thornhill's approach, 
essentially since in this special case Theorem~\ref{thm:criteria} is a straightforward consequence of
Kuyk's theorem and the diamond theorem.
We show in Section~\ref{sec:prfGR} that Theorem~\ref{thm:main}, and hence Theorem~\ref{conj:Jarden}, follows rather simply from Theorem~\ref{thm:criteria} and the theorem of Larsen and Pink.
Our proof of Theorem~\ref{thm:criteria}, which takes up Sections \ref{sec:gds} and \ref{sec:criterion} of this paper, 
utilizes the twisted wreath product approach that Haran developed to prove his diamond theorem. 

Theorem~\ref{thm:criteria} has many other applications, some of which are presented in Section~\ref{sec:appl}.  
For example we show that if $L$ is the compositum of all degree $d$ extensions of $\mathbb{Q}$, for some fixed $d$, then every subfield of $L$ is Hilbertian (Theorem~\ref{thm:bd}). 
A similar application is given when $L$ is the compositum of the fixed fields of all $ \ker \bar\rho$, where $\bar\rho$ runs over all \emph{finite} representations of $\Gal(\mathbb{Q})$ of fixed dimension $d$ (Theorem~\ref{thm:allrep}).  
We also include one application to the theory of free profinite groups (Theorem~\ref{thm:free}).

\section{Groups of finite abelian-simple length}\label{sec:gds}

\subsection{Basic theory}
Let $G$ be a profinite group. We define the \textbf{generalized derived subgroup} $D(G)$ of $G$ as the intersection of all open normal subgroups $N$ of $G$ with $G/N$ either abelian or simple. 
The \textbf{generalized derived series} of $G$, 
\[
G=G^{(0)} \geq G^{(1)} \geq G^{(2)}\geq \cdots,
\]
is defined inductively by $G^{(0)}=G$ and $G^{(i+1)} = D(G^{(i)})$ for $i\geq0$.

\begin{lem}\label{lem:basic_prop_gds}
Let $G$ be a profinite group with generalized derived series $G^{(i)}$, $i=0,1,2,\ldots$. Then
$G^{(i)}\lhd G$ for each $i\geq 0$, and if $G^{(i)}\neq 1$, then $G^{(i+1)}\neq G^{(i)}$.
\end{lem}

\begin{proof}
Since $G^{(i)}$ is characteristic in $G^{(i-1)}$ and since by induction $G^{(i-1)}\lhd G$,  the first assertion follows.
If $G^{(i)}\neq 1$, then $G^{(i)}$ has a nontrivial finite simple quotient. Hence $G^{(i+1)} = D(G^{(i)})\neq G^{(i)}$, as claimed in the second assertion.
\end{proof}

We define the \textbf{abelian-simple length} of a profinite group $G$, denoted by $l(G)$, to be the smallest integer $l$ for which $G^{(l)}=1$. If $G^{(i)}\neq 1$ for all $i$, we set $l(G) =  \infty$.
We say that $G$ is \textbf{of finite abelian-simple length} if $l(G)<\infty$.

\begin{lem}\label{lem:log_length}
Let $G$ be a finite group. Then $l(G) \leq \log_2(|G|)<\infty$.
\end{lem}

\begin{proof}
By Lemma~\ref{lem:basic_prop_gds}, if $G^{(i)}\neq 1$, then $[G^{(i)}:G^{(i+1)}] \geq 2$. Hence $|G| = \prod_{i=1}^{l} [G^{(i-1)} :G^{(i)}] \geq 2^{l}$, where $l=l(G)$.
\end{proof}

\begin{ex}
If $G$ is pro-solvable, then the generalized derived series coincides with the derived series of $G$. 
In particular, such $G$ is solvable if and only if $G$ is of finite abelian-simple length.
\end{ex}

We will need the following well-known result.

\begin{lem}\label{lem:directproductsimplegroups}
Let $G=A\times\prod_{i\in I} S_i$ be a profinite group, where $A$ is abelian and each $S_i$ is a non-abelian finite simple group. 
If $N$ is a closed normal subgroup of $G$, then
$N=(N\cap A)\times\prod_{i\in J} S_i$ for some subset $J\subseteq I$. 
In particular, $G/N\cong (AN/N)\times\prod_{i\in I\smallsetminus J} S_i$.
\end{lem}

\begin{proof}
The proof of \cite[Lemma 25.5.3]{FriedJarden} goes through almost literally.
\end{proof}

\label{def:D0}
For a profinite group $G$ let us denote by $D_0(G)$ the intersection of all maximal open normal subgroups $N$ of $G$ such that
$G/N$ is not abelian. Then $D(G)=D_0(G)\cap G'$, where $G'$ is the commutator subgroup of $G$.

\begin{lem}\label{lem:basic_prop_sigma-derived}
Let $G$ be a profinite group. Then $D(G)$ is the smallest closed normal subgroup of $G$ such that $G/D(G)$
is isomorphic to a direct product of finite simple groups and abelian groups.
\end{lem}

\begin{proof}
By \cite[Lemma 18.3.11]{FriedJarden}, $G/D_0(G)$ is a product of finite non-abelian simple groups,
and if $N$ is a maximal closed normal subgroup of $G$ containing $D_0(G)$, then $G/N$ is non-abelian.
It follows that $D_0(G)G'=G$, so $G/D(G)=(G/G')\times(G/D_0(G))$ is of the asserted form.

Conversely, assume that $N$ is a closed normal subgroup of $G$ with $G/N=\prod_{i\in I}S_i$, with $S_i$ either abelian or finite simple.
Let $\pi_j:\prod_{i\in I}S_i\rightarrow S_j$ be the projection map and $\pi:G\rightarrow G/N$ the quotient map.
Then $N=\bigcap_{j\in I}\ker(\pi_j\circ\pi)\geq D(G)$, since $G/\ker(\pi_j\circ\pi)\cong S_j$.
\end{proof}

\begin{lem}\label{lem:subnormal}
Let $G$ be a profinite group. Then $G$ is of finite abelian-simple length if and only if there exists a subnormal series
$$
 G=G_0\rhd G_1\rhd\dots\rhd G_r=1
$$ 
of closed subgroups of $G$
with all factors $G_{i-1}/G_{i}$ either abelian or a product of finite simple groups.
\end{lem}

\begin{proof}
If $G_i$, $i=0,\dots,r$ is a subnormal series of $G$ of length $r$ of the asserted form, then $D(G)\leq G_1$ by Lemma \ref{lem:basic_prop_sigma-derived}.
Thus, $G_i\cap D(G)$, $i=1,\dots,r$ is a subnormal series of $D(G)$ of length $r-1$, so induction on $r$ shows that $l(D(G))<\infty$, and hence $l(G)<\infty$.

Conversely, if $l(G)<\infty$, then, by Lemma \ref{lem:basic_prop_sigma-derived}, $G^{(i)}$, $i=0,\dots,l(G)$ can easily be refined to a subnormal series of the asserted form.
\end{proof}

\begin{lem}\label{lem:basicpropertiesofmSigma}
\begin{enumerate}
\item \label{cond1:basicpropertiesofmSigma} If $\alpha\colon G \to H$ is an epimorphism of profinite groups, then $\alpha(G^{(i)})$, $i=0,1,2\ldots$ is the generalized derived series of $H$.  In particular, $l(H)\leq l(G)$.
\item \label{cond2:basicpropertiesofmSigma} If $N$ is a closed normal subgroup of a profinite group $G$, then $N^{(i)}\leq G^{(i)}$ for each $i$. In particular, $l(N) \leq l(G)$.
\item \label{cond3:basicpropertiesofmSigma} 
If $\mathcal{N}$ is a downward directed family of closed normal subgroups of $G$ with $\bigcap_{N\in\mathcal{N}}N=1$, 
then $G^{(i)}=\varprojlim_{N\in\mathcal{N}} (G/N)^{(i)}$ for each $i$. 
In particular, $l(G)=\sup_{N\in\mathcal{N}} l(G/N)$.
\item \label{cond4:basicpropertiesofmSigma}
Let $\alpha\colon G\to K$  and $\beta\colon H\to K$ be epimorphisms of profinite groups,
and let $G\times_KH$ be the corresponding fiber product. Then $(G\times_{K} H)^{(i)}\leq G^{(i)}\times_{K^{(i)}} H^{(i)}$ for each $i$. In particular, $l(G\times_{K} H) \leq \max(l(G),l(H))$.
\end{enumerate}
\end{lem}

\begin{proof}
Recall that the Melnikov subgroup $M(G)$ of a profinite group $G$ is defined as the intersection of all maximal open normal subgroups of $G$. Thus $D(\Gamma) = M(\Gamma)\cap \Gamma'$, for any profinite group $\Gamma$.

We start with Assertion~\ref{cond1:basicpropertiesofmSigma}. Since $\alpha$ is surjective, it follows that $\alpha(G') = H'$ and $\alpha(M(G)) = M(H)$ \cite[Lemma 25.5.4]{FriedJarden}. Thus $\alpha(D(G)) \subseteq D(H)$.
Since $H/\alpha(D(G))$ is a quotient of $G/D(G)$, Lemma \ref{lem:basic_prop_sigma-derived} and Lemma \ref{lem:directproductsimplegroups} 
together give that $\alpha(D(G))\supseteq D(H)$. 
So, $\alpha(D(G))=D(H)$. By induction we are done.

To get Assertion~\ref{cond2:basicpropertiesofmSigma} note that $N'\leq G'$ and that $M(N)\leq M(G)$ by \cite[Lemma 25.5.1]{FriedJarden}. Therefore $D(N)\leq D(G)$. Since $D(N)$ is characteristic in $N$, we get that $D(N)\lhd D(G)$. Therefore, by induction, $N^{(i)}\leq G^{(i)}$ for every $i$.

Let us prove Assertion~\ref{cond3:basicpropertiesofmSigma}. 
Since $\mathcal{N}$ is directed and $\bigcap_{N\in\mathcal{N}}N=1$, $G=\varprojlim_{N\in\mathcal{N}} G/N$.
Assertion~\ref{cond1:basicpropertiesofmSigma} implies that
if $N\in\mathcal{N}$ and $\pi_N:G\rightarrow G/N$ is the quotient map, $\pi_N(G^{(i)})=(G/N)^{(i)}$.
Thus, $G^{(i)}=\varprojlim_{N\in\mathcal{N}} (G/N)^{(i)}$, cf.~\cite[Corollary 1.1.8a]{RibesZalesskii}.

Finally we get to Assertion~\ref{cond4:basicpropertiesofmSigma}.
Set $L=G\times_K H = \{(g,h) \colon \alpha(g)=\beta(h)\} \leq G\times H$. Let $M=1 \times \ker \beta$ and $N = \ker\alpha \times 1$. Then 
$NM= \ker\alpha\times\ker\beta\cong N\times M$, $L/M = G$, and $L/N=H$.
Thus,
if $\mathcal{U}$ denotes the set of closed normal subgroups $U$ of $L$ with $L/U$ either abelian or finite simple, then
\begin{eqnarray*}
D(L) \;=\; \bigcap_{U\in\mathcal{U}} U \;\leq\; \big(\bigcap_{\substack{U\in\mathcal{U}\\M\leq U}}U \big)\cap \big( \bigcap_{\substack{U\in\mathcal{U}\\N\leq U}}U\big) \;=\; D(G) \times_{D(K)}D(H).
\end{eqnarray*}
The last equality holds since $\alpha(D(G))=\beta(D(H))=D(K)$ by Assertion~1.
By induction, 
$L^{(i)}\leq G^{(i)}\times_{K^{(i)}} H^{(i)}$.
\end{proof}

\begin{prop}\label{prop:lengthN}
If $\mathcal{N}$ is a family of closed normal subgroups of $G$ with $\bigcap_{N\in\mathcal{N}}N=1$,
then $l(G)=\sup_{N\in\mathcal{N}}l(G/N)$.
\end{prop}

\begin{proof}
Let $\mathcal{N}'$ be the family of finite intersections of elements of $\mathcal{N}$.
By Lemma \ref{lem:basicpropertiesofmSigma}.3, $l(G)=\sup_{N\in\mathcal{N}'}l(G/N)$.
If $N_1,N_2\in\mathcal{N}$, then $G/(N_1\cap N_2)\cong (G/N_1)\times_{G/(N_1N_2)}(G/N_2)$,
so $\sup_{N\in\mathcal{N}'}l(G/N)=\sup_{N\in\mathcal{N}}l(G/N)$ by Lemma~\ref{lem:basicpropertiesofmSigma}.4.
\end{proof}

We already saw that the class of profinite groups of finite abelian-simple length is closed under taking
quotients and normal subgroups. In fact, it is also closed under forming group extensions:

\begin{prop}\label{prop:lengthextension}
Let $N$ be a closed normal subgroup of $G$. Then $l(G)\leq l(N)+l(G/N)$.
\end{prop}

\begin{proof}
Let $m=l(G/N)$, $n=l(N)$, and let $\pi:G\rightarrow G/N$ be the quotient map.
By Lemma \ref{lem:basicpropertiesofmSigma}.1, $\pi(G^{(m)})=(G/N)^{(m)}=1$, so $G^{(m)}\leq N$.
Since $G^{(m)}$ is normal in $G$ and hence in $N$, 
Lemma \ref{lem:basicpropertiesofmSigma}.2 implies that $G^{(m+n)}=(G^{(m)})^{(n)}\leq N^{(n)}=1$.
Hence, $l(G)\leq m+n$.
\end{proof}

\subsection{Twisted wreath products}
Let $A$ and $G_0\leq G$ be finite groups together with a (right) action of $G_0$ on $A$. 
The set of $G_0$-invariant functions from $G$ to $A$,
$$
 \Ind_{G_0}^G (A) = \left\{ f\colon G\to A \;\mid\; f(\sigma\tau)=f(\sigma)^\tau,\ \forall \sigma\in G, \tau\in G_0\right\},
$$ 
forms a group under pointwise multiplication, on which $G$ acts from the right
by $f^\sigma(\tau)= f(\sigma\tau)$, for all $\sigma,\tau\in G$. The \textbf{twisted wreath product} is defined to be the semidirect product
\[
 A\wr_{G_0} G = \Ind_{G_0}^G(A) \rtimes G,
\]
cf.~\cite[Definition 13.7.2]{FriedJarden}.
The following observation will be used several times.
\begin{lem}\label{lem:exact}
Let $A$ and $G_0\leq G$ be finite groups together with an action of $G_0$ on $A$, and let $A_0$ be a normal subgroup of $A$ that is $G_0$-invariant. Then $G_0$ acts on $A/A_0$ and
\[
\xymatrix@1{1\ar[r]&\Ind_{G_0}^G(A_0) \ar[r]&A\wr_{G_0} G\ar[r]^-{\alpha}& (A/A_0)\wr_{G_0} G \ar[r]&1, }
\]
is an exact sequence of finite groups. Here  $\alpha(f,\sigma) = (\bar{f},\sigma)$, where $\bar{f}(\tau) = f(\tau)A_0$, for $\tau\in G$.
\end{lem}

\begin{proof}
$G_0$ acts on $A/A_0$ by $(aA_0)^\sigma = a^\sigma A_0$. To see that $\alpha$ is a homomorphism note that $\bar{f}^\sigma(\tau) = f(\sigma\tau) A_0 = f^{\sigma}(\tau) A_0 = \overline{f^\sigma}(\tau)$. It is trivial that $\alpha$ is surjective and that $\ker(\alpha) = \Ind_{G_0}^G(A_0)$.
\end{proof}

The main objective of this section is to show that the abelian-simple length grows in wreath products:

\begin{prop}\label{prop:m-sigmalength}
Let $m\in \NN$, let $A$ be a nontrivial finite  group,  and let $G_0\leq G$ be finite groups together with an action of $G_0$ on $A$. Assume that $$[G^{(m)}G_0:G_0] > 2^m.$$ Then $$(A\wr_{G_0} G)^{(m+1)} \cap \Ind_{G_0}^G(A)\neq 1.$$
\end{prop}

\begin{rem}
We do not know whether the assumption $[G^{(m)}G_0:G_0]>2^m$ can be replaced by the weaker condition $G^{(m)}\not\subseteq G_0$.
Our proof makes essential use of the stronger assumption only once (namely in the ``Second Case'' of Lemma \ref{lem:abeliancase}).
\end{rem}

\subsection{Proof of Proposition \ref{prop:m-sigmalength}}

The rest of this section is devoted to proving Proposition \ref{prop:m-sigmalength}.
It is rather technical and the auxiliary statements will 
not be used anywhere else in this paper.
First we deal with several special cases depending on $A$ and on the action of $G_0$, and then deduce the proposition.

\subsubsection*{Direct product of non-abelian simple groups}

\begin{lem}\label{lem:simple_nonabelian}
Under the assumptions of Proposition \ref{prop:m-sigmalength},
let $S$ be a non-abelian finite simple group, and assume that $A\cong \prod_{j=1}^n S_j$ with $n\geq 1$ and $S_j\cong S$ for all $j$. 
Let $H= A\wr_{G_0} G$. Then $H^{(m+1)} = \Ind_{G_0}^G (A)\rtimes G^{(m+1)}$. In particular, $H^{(m+1)}\cap \Ind_{G_0}^G (A)\neq 1$.
\end{lem}

\begin{proof}
Let $L=\Ind_{G_0}^G(A)$.  
We prove by induction on $i$ that $H^{(i)} = L\rtimes G^{(i)}$, for every $i=0,\ldots, m+1$. For $i=0$ the assertion is trivial.

 Next assume that $i<m+1$ and $H^{(i)} = L \rtimes G^{(i)}$. It is clear that $H^{(i+1)}\leq L\rtimes G^{(i+1)}$, since the former is the intersection of all normal subgroups of $ L \rtimes G^{(i)}$ whose quotient is either simple or abelian and the latter is the intersection of the subfamily of those normal subgroups that contain $L$. Therefore it suffices to show that if $U$ is a normal subgroup of $L\rtimes G^{(i)}$ with $(L\rtimes G^{(i)})/U$ either simple or abelian, then $L\leq U$.

Assume the contrary, so  $L\cap U$ is a proper normal subgroup of $L$. The quotient $L/(L\cap U)\cong LU/U$ is either simple or abelian, as a nontrivial normal subgroup of the simple resp.~abelian group $(L\rtimes G^{(i)})/ U$. Since $L$ is a direct product of copies of $A$, hence of $S$, $L/(L\cap U)$ is not abelian.  Lemma~\ref{lem:directproductsimplegroups} implies that $L\cap U$ is  the kernel of one of the projections $L\to A\to S_j$.  
 More precisely, there exist $\sigma\in G$ and $1\leq j\leq n$ such that if we denote by $\pi_j\colon A\to S_j$ the projection map, then 
$L\cap U=\{f\in \Ind_{G_0}^G (A)\mid \pi_j(f(\sigma))=1\}$.

The assumption $[G^{(m)}G_0:G_0] > 2^m$ implies that $G^{(i)}\not\subseteq G_0$, since $i\leq m$.
So since $G^{(i)}$ is normal in $G$, there exists $\tau\in G^{(i)}\smallsetminus (G_0)^{\sigma^{-1}}$. 
It follows that
\begin{eqnarray*}
L\cap U&=&(L\cap U)^\tau \;=\;  \{f\in \Ind_{G_0}^G (A)\mid \pi_j(f(\sigma))=1\}^\tau \\
&=&  \{f\in \Ind_{G_0}^G (A)\mid \pi_j(f(\tau^{-1} \sigma))=1\}. 
\end{eqnarray*}
Since $\tau^{-1}\sigma G_0\neq\sigma G_0$ by the choice of $\tau$, there exists $f\in L$ such that $f(\tau^{-1} \sigma) \neq 1$ but $f(\sigma)= 1$. This implies that $f\in (L\cap U)\smallsetminus (L\cap U)^{\tau}=\emptyset$ --  a contradiction.
Thus $L\leq U$, as needed.
\end{proof}

\subsubsection*{Nontrivial irreducible representation}
We say that a representation $V$ of $G_0$ is \textbf{nontrivial} if there exist $v\in V$ and $\sigma\in G_0$ such that $v^\sigma\neq v$.  

\begin{lem}\label{lem:nontrivialaction}
Let $G_0\leq G$ be finite groups, let $p$ be a prime number and let $A$ be a nontrivial finite irreducible $\mathbb{F}_p$-representation of $G_0$. Let $H=A\wr_{G_0} G$. Then $H' = \Ind_{G_0}^G(A) \rtimes G'$.
\end{lem}

\begin{proof}
Let $A_0 $ be the subspace of $A$ spanned by $\{a^\sigma - a \mid a\in A, \sigma\in G_0\}$. 
Since $(a^{\sigma} - a)^{\tau} = (a^{\tau})^{\tau^{-1}\sigma\tau}-a^{\tau}\in A_0$
for all $\tau\in G_0$, we get that $A_0$ is $G_0$-invariant.  Since the action of $G_0$ is nontrivial, $A_0\neq 0$. The assumption that $A$ is an irreducible representation then implies that $A_0=A$.

For $a\in A$ and $\sigma \in G_0$, let $\tilde{f}_{a,\sigma}\in\Ind_{G_0}^G(A)$ be defined by $\tilde{f}_{a, \sigma}(\tau) = a^{\sigma\tau} - a^\tau$ if $\tau\in G_0$ and $\tilde{f}_{a,\sigma}(\tau)=1$ if $\tau\in G\smallsetminus G_0$. Then $\tilde{F}= \{\tilde{f}_{a,\sigma} \mid a\in A, \sigma\in G_0\}$ generates the subgroup $\{f\mid f(\tau)=1,\ \forall \tau\in G\smallsetminus G_0\}$ of $\Ind_{G_0}^G(A)$. Hence $\tilde{F}$ generates $\Ind_{G_0}^G(A)$ as a normal subgroup of $H=A\wr_{G_0} G$.

But $\tilde{f}_{a,\sigma} = [f_a,\sigma]$, where $f_a(\tau) = a^{\tau}$ if $\tau\in G_0$ and $f(\tau)=1$ otherwise.  So $\tilde{f}_{a,\sigma}\in H'$ and hence $\Ind_{G_0}^G(A)\leq H'$. So $\Ind_{G_0}^G(A)\rtimes G'\leq H'$. On the other hand, since $H/(\Ind_{G_0}^G(A)\rtimes G') \cong G/G'$ is abelian, $\Ind_{G_0}^G(A)\rtimes G'\geq H'$.
Thus, $H'=\Ind_{G_0}^G(A)\rtimes G'$.
\end{proof}

\subsubsection*{Trivial representation}

Let $p$ be a prime number and $G$ a finite group acting on a finite set $X$.
Then the group ring $\mathbb{F}_p[G]$ acts on the $\mathbb{F}_p$-vector space $V_0(G,X)=(\mathbb{F}_p)^X$ of functions from $X$ to $\mathbb{F}_p$. 
For $i\geq0$ we let
$V_{i+1}(G,X)$ be the subspace of $V_{i}(G,X)$ generated by the elements $f^{1-g}=f-f^g$, where $f\in V_i(G,X)$ and $g\in G^{(i)}$.

\begin{lem}\label{prop:trvrep}
Let $G$ be a finite group acting on a finite set $X$, and let $m\in\mathbb{N}$. 
Assume there exists $x\in X$ such that $|G^{(m)}x|>2^m$. Then $V_{m+1}(G,X)\neq 0$.
\end{lem}

\begin{proof}
For $f\in V_0(G,X)$ let ${\rm supp}(f)=\{y\in X:f(y)\neq 0\}$. 
Let $f_0\in V_0(G,X)$ be defined by $f_0(x)=1$ and $f_0(y)=0$ for all $y\neq x$. 
We construct inductively a sequence $g_1,...,g_{m+1}\in G^{(m)}$ such that for each $0\leq k\leq m+1$, 
$$
 f_k:=f_0^{(1-g_1)\cdots(1-g_k)}\in V_k(G,X)
$$ 
satisfies $0<|{\rm supp}(f_k)|\leq 2^k$. 
For $k=0$, there is nothing to show. 
Let $0<k\leq m$ and assume $g_1,\dots,g_k$ are already constructed. 
Since $\emptyset\neq {\rm supp}(f_k)\subseteq G^{(m)}x$ and $|{\rm supp}(f_k)|\leq 2^k<|G^{(m)}x|$, there exists $g_{k+1}\in G^{(m)}$ with ${\rm supp}(f_k^{g_{k+1}})\not\subseteq{\rm supp}(f_k)$, hence $f_{k+1}=f_k-f_k^{g_{k+1}}\neq 0$. 
Clearly, $|{\rm supp}(f_{k+1})|\leq 2|{\rm supp}(f_k)|\leq 2^{k+1}$, concluding the induction step.
\end{proof}

\subsubsection*{Abelian group}

\begin{lem}\label{lem:abeliancase}
Let $p$ be a prime number. Let $n\geq 0$, let $k>0$, let $G_0\leq G$ be finite groups, 
and let $G_0$ act on $A=(\mathbb{F}_p)^k$.
If $n>0$, assume in addition that $[G^{(n-1)}G_0:G_0]>2^{n-1}$.
Then $(A\wr_{G_0} G)^{(n)}\cap{\rm Ind}_{G_0}^G(A)\neq 1$.
\end{lem}

\begin{proof}
We prove the lemma by induction on $n$. For $n=0$, the claim is obvious since $k>0$.
Therefore, assume that $n>0$ and that 
the statement of the lemma holds for all smaller $n$ (for arbitrary groups $G_0,G$ and arbitrary $k>0$).

Let $V_0$ be a maximal $G_0$-invariant proper subspace of $A$. 
Then $V=A/V_0\neq 0$ is an irreducible $\mathbb{F}_p$-representation of $G_0$.
By Lemma~\ref{lem:exact}, $V\wr_{G_0} G$ is a quotient of $A\wr_{G_0} G$ and
$\Ind_{G_0}^G(A)$ is the preimage of $\Ind_{G_0}^G(V)$. Thus, by Lemma \ref{lem:basicpropertiesofmSigma}.1 it suffices to show that $(V\wr_{G_0} G)^{(n)}\cap \Ind_{G_0}^G(V)\neq 1$.  

To this end set $H=V\wr_{G_0} G$ and $L=\Ind_{G_0}^G(V)$. If $N$ is a normal subgroup of $H$ with $H/N$ simple non-abelian, then $L/(N\cap L)=LN/N$ is an abelian normal subgroup of $H/N$, hence $L/(N\cap L)=1$, so $L\leq N$. Therefore $D_0(H) = L\rtimes D_0(G)$ (cf.~p.~\pageref{def:D0}) and  
\begin{equation}\label{eq:h1sigma}
D(H) = D_0(H) \cap H' =   (L\rtimes D_0(G))\cap H'.
\end{equation}
We distinguish between two cases:

\noindent{\sc First Case}: $G_0$ acts nontrivially on $V$.

By Lemma~\ref{lem:nontrivialaction} we have $H'=L\rtimes G'$. Plugging this into \eqref{eq:h1sigma} we get that $D(H) = L\rtimes D(G)$. 
Let $\tilde{G}=D(G)$ and $\tilde{G}_0=\tilde{G}\cap G_0$.
If $n>1$, then 
$$
 [\tilde{G}^{(n-2)}\tilde{G}_0:\tilde{G}_0]=[\tilde{G}^{(n-2)}:\tilde{G}^{(n-2)}\cap\tilde{G}_0]=[G^{(n-1)}:G^{(n-1)}\cap G_0]>2^{n-1}>2^{n-2}.
$$ 
Applying the induction hypothesis to $\tilde{G}_0\leq\tilde{G}$ and $V$ gives that
$(V\wr_{\tilde{G}_0}\tilde{G})^{(n-1)}\cap{\rm Ind}_{\tilde{G}_0}^{\tilde{G}}(V)\neq 1$.
The epimorphism $\pi:L\rtimes  D(G)\rightarrow V\wr_{\tilde{G}_0}\tilde{G}$, $(f,\sigma)\mapsto(f|_{\tilde{G}},\sigma)$, 
restricts by Lemma \ref{lem:basicpropertiesofmSigma}.1 to
an epimorphism $(V\wr_{G_0}G)^{(n)}=(L\rtimes  D(G))^{(n-1)}\rightarrow(V\wr_{\tilde{G}_0}\tilde{G})^{(n-1)}$,
and $\pi^{-1}({\rm Ind}_{\tilde{G}_0}^{\tilde{G}}(V))={\rm Ind}_{G_0}^G(V)$.
The claim follows.

\noindent{\sc Second Case}: $G_0$ acts trivially on $V$. 

Since $V$ is an irreducible $\mathbb{F}_p$-representation of $G_0$, $V\cong\mathbb{F}_p$. 
Note that $[f,\sigma] = f^{\sigma}-f$, for all $f\in \Ind_{G_0}^G(V)$ and $\sigma\in G$.  
Let $V_i=V_i(G,G/G_0)$.
Since the action of $G_0$ on $V$ is trivial we can identify $\Ind_{G_0}^G(V)$ with $V_0=(\mathbb{F}_p)^{G/G_0}$.
Hence, $[\Ind_{G_0}^G(V), G] = V_1$. 
Thus $H'\geq V_1\rtimes G'$. 
Since $(V_0\rtimes G)/(V_1\rtimes G')\cong (V_0/V_1)\times(G/G')$ is abelian,
we get that $H'= V_1\rtimes G'$.
Plugging this into \eqref{eq:h1sigma} we get that $D(H)=V_1\rtimes D(G)$. Inductively, if $H^{(i)} = V_{i}\rtimes G^{(i)}$, then as in the paragraph preceding \eqref{eq:h1sigma}, $D_0(H^{(i)}) = V_i \rtimes D_0(G^{(i)})$. Since $[f,\sigma]=f^{\sigma-1}\in V_{i+1}$ for all $f\in V_i$ and $\sigma\in G^{(i)}$, we have $(H^{(i)})' = V_{i+1} \rtimes (G^{(i)})'$ as above. So 
\[
H^{(i+1)}= D(H^{(i)}) = D_0(H^{i})\cap (H^{(i)})' = V_{i+1} \rtimes G^{(i+1)}.
\] 
In particular, $H^{(n)} \cap \Ind_{G_0}^G(V)=V_{n}$.
Finally, since we assume that $[G^{(n-1)} G_0 : G_0]>2^{n-1}$, Lemma~\ref{prop:trvrep} gives that $V_{n}\neq 0$, proving the claim.
\end{proof}

\subsubsection*{Proof of Proposition~\ref{prop:m-sigmalength}}
Assume that $G_0\leq G$ are finite groups, $G_0$ acts on the nontrivial finite group $A$, and $[G^{(m)}G_0:G_0]>2^m$.
We claim that $(A\wr_{G_0} G)^{(m+1)} \cap \Ind_{G_0}^G(A)\neq 1$.

Let $S$ be a simple quotient of $A$, and let $A_0$ be the intersection of all normal subgroups of $A$ with $A/N \cong S$. Then $A_0$ is characteristic in $A$, hence $G_0$-invariant, and $\bar{A}=A/A_0$ is isomorphic to a nonempty direct product of copies of $S$
(this is clear if $S$ is abelian; see \cite[Lemma 8.2.3]{RibesZalesskii} for the case where $S$ is non-abelian).
By Lemma~\ref{lem:exact}, $\bar{A}\wr_{G_0} G$ is a quotient of $A\wr_{G_0} G$ and the preimage of $\Ind_{G_0}^G(\bar{A})$ is $\Ind_{G_0}^G(A)$. Thus, by Lemma~\ref{lem:basicpropertiesofmSigma}.1, it suffices to show that $(\bar{A}\wr_{G_0} G)^{(m+1)} \cap \Ind_{G_0}^G(\bar{A})\neq 1$.

If $\bar{A}$ is non-abelian, then the assertion follows from Lemma~\ref{lem:simple_nonabelian}. If $\bar{A}$ is abelian, then the assertion follows from Lemma~\ref{lem:abeliancase}.
\qed

\section{Hilbertianity criterion}
\label{sec:criterion}

We shall use the twisted wreath product approach of Haran, which we briefly recall for the reader's convenience. 

We say that a tower of fields  $K\subseteq E_0\subseteq  E\subseteq  F\subseteq \hat{F}$ \textbf{realizes} a twisted wreath product $A\wr_{G_0} G$ if $\hat{F}/K$ is a Galois extension with Galois group isomorphic to $A\wr_{G_0} G$ and the tower of fields corresponds to the following subgroup series:
\[
A\wr_{G_0} G \geq \Ind_{G_0}^G(A) \rtimes G_0 \geq \Ind_{G_0}^G(A) \geq \{f\in \Ind_{G_0}^G(A)\mid f(1)=1\}\geq 1.
\]
This definition coincides with \cite[Remark~1.2]{HaranDiamond}.

\begin{thm}[Haran {\cite[Theorem 3.2 and the remark just before it]{HaranDiamond}}]\label{thm:haran}
Let $M$ be a separable algebraic extension of a Hilbertian field $K$. Suppose that for every $\alpha$ in $M$ and every $\beta$
in the separable closure $M_s$ of $M$ there exist
\begin{enumerate}
 \item a finite Galois extension $E$ of $K$ that contains $\beta$; let $G = \Gal(E/K)$;
 \item a field $E_0$ such that $K\subseteq E_0\subseteq M\cap E$ and $E_0$ contains $\alpha$; let $G_0=\Gal(E/E_0)$;
 \item a Galois extension $N$ of $K$ that contains both $M$ and $E$,
\end{enumerate}
such that for every nontrivial finite group $A$ and every action of $G_0$ on $A$ there is no realization $K\subseteq E_0\subseteq  E\subseteq  F\subseteq \hat{F}$ of $A\wr_{G_0} G$ with $\hat{F}\subseteq N$. 
Then $M$ is Hilbertian.
\end{thm}

We say that a separable algebraic extension $M/K$ is \textbf{of finite abelian-simple length} 
if $\Gal(L/K)$ is, where $L$ denotes the Galois closure of $M/K$.

\begin{thm}\label{thm:bdgdlimplieshilb}
Let $M$ be a separable algebraic extension of a Hilbertian field $K$ of finite abelian-simple length. Then $M$ is Hilbertian.
\end{thm}

\begin{proof}
Let $L$ be the Galois closure of $M/K$. Let $\Gamma = \Gal(L/K)$ and let $\Gamma^{(i)}$, $i=0,1,2,\ldots $ be the generalized derived series of $\Gamma$. 
By assumption there exists a minimal $m\geq0$ such that $\Gamma^{(m+1)}=1$. 
Let $\Gamma_0 = \Gal(L/M)$ and for each $i$ denote by $L^{(i)}$ the fixed field of $\Gamma^{(i)}$ in $L$.

If $[\Gamma_0\Gamma^{(m)} : \Gamma_0]<\infty$, then, by Galois correspondence, $M$ is a finite extension of $U = M\cap L^{(m)}$. Note that if $\hat{U} $ is the Galois closure of $U/K$, then  $\hat{U}\subseteq L^{(m)}$ and thus $\Gal(\hat{U}/K)$ is a quotient of $\Gamma/\Gamma^{(m)}$. Thus $\Gal(\hat{U}/K)^{(m)}$ is a quotient of $(\Gamma/\Gamma^{(m)})^{(m)} = \Gamma^{(m)}/\Gamma^{(m)}= 1$ and therefore trivial (Lemma~\ref{lem:basicpropertiesofmSigma}). Therefore induction on $m$ implies that $U$ is Hilbertian, and hence $M$ is Hilbertian as a finite extension of $U$,
see \cite[Proposition 12.3.3]{FriedJarden}.

Therefore we may assume that  $[\Gamma_0\Gamma^{(m)} : \Gamma_0]=\infty$, i.e.\ $[M:M\cap L^{(m)}] = \infty$.
To prove that $M$ is Hilbertian we apply Theorem~\ref{thm:haran}. Let $\alpha\in M$ and $\beta\in M_s$.
Since $M/M\cap L^{(m)}$ is infinite there exists a finite Galois extension $E/K$ such that $\alpha,\beta\in E$ and $[E\cap M : E\cap M \cap L^{(m)}] > 2^m$.

Let $E_0 = E\cap M$, $G = \Gal(E/K)$, $G_0 = \Gal(E/E_0)$, and let $G^{(i)}$, $i=0,1,2,\ldots $ be the generalized derived series of $G$. Note that $\alpha \in E_0$. 
We also let $N = EL$ and $A$ a nontrivial group on which $G_0$ acts. 
By Theorem~\ref{thm:haran} it suffices to prove that there is no realization $K\subseteq E_0\subseteq  E\subseteq  F\subseteq \hat{F}$ of $A\wr_{G_0} G$ with $\hat{F}\subseteq N$. Assume the contrary and identify $\Gal(\hat{F}/K)$ and $A\wr_{G_0} G$.

Let $\check{E} = E\cap L$, $\check{G}=\Gal(\check{E}/K)$, and let $\phi\colon \Gamma \to \check{G}$ and $\psi \colon G\to \check{G}$ be the corresponding restriction maps. 

\[
\xymatrix@C=40pt{
&L^{(m)}\ar@{-}[r]\ar@/^1pc/@{.}[rr]^{\Gamma^{(m)}}
&L^{(m)}M\ar@{-}[r]
&L\ar@{-}[r]
&N
\\
&L^{(m)}\cap M\ar@{-}[r]\ar@{-}[u]
&M\ar@{-}[u]\ar@{.}[ur]|{\Gamma_0}
\\
K\ar@{-}[r]\ar@/_1pc/@{.}[rrrr]_{G}
&E_0\cap L^{(m)}\ar@{-}[u]\ar@{-}[r]^-{>2^m}
&E_0\ar@{-}[u]\ar@{-}[r]
&\tilde{E}\ar@{-}[uu]\ar@{-}[r]
&E\ar@{-}[uu]
}
\]
By Lemma~\ref{lem:basicpropertiesofmSigma}.1,
\begin{eqnarray*}
\check{G}^{(m)} &=& \phi(\Gamma^{(m)}) \;=\; \Gal(\check{E}/ L^{(m)} \cap \check{E}),\\
\check{G}^{(m)} &=& \psi(G^{(m)}) \;=\; \Gal(\check{E}/ E^{(m)} \cap \check{E}),
\end{eqnarray*}
where $E^{(m)}$ is the fixed field of $G^{(m)}$ in $E$. Thus $E^{(m)}\cap \check{E} = L^{(m)}\cap \check{E}$. Since $E\cap M = \check{E}\cap M$  we have 
\[
E\cap M \cap E^{(m)} = \check{E} \cap M \cap E^{(m)} = M \cap E^{(m)}\cap  \check{E} =  M\cap L^{(m)} \cap \check{E} = \check{E}\cap M \cap L^{(m)}=E\cap M \cap L^{(m)}.
\]
 So
\begin{eqnarray*}
[G^{(m)} G_0 : G_0] &=& [E\cap M : E\cap M \cap E^{(m)}] \\
&=& [E \cap M : E\cap M \cap L^{(m)}] \;>\; 2^m.
\end{eqnarray*}
Proposition~\ref{prop:m-sigmalength} gives that
\begin{equation*}\label{eq:mSigmaisHextb}
(A\wr_{G_0}G)^{(m+1)}\cap \Ind_{G_0}^G(A) \neq 1.
\end{equation*}
Let $\tau \in (A\wr_{G_0}G)^{(m+1)}\cap \Ind_{G_0}^G(A) $ be nontrivial. Lift $\tau$ to $T\in \Gal(N/K)^{(m+1)}$ (Lemma~\ref{lem:basicpropertiesofmSigma}).  But $T|_{L} \in \Gal(L/K)^{(m+1)}=1$ by the same lemma. Since $\tau\in \Ind_{G_0}^G(A) = \Gal(\hat{F}/E)$, it follows that $T|_{E}=1$. But then $T = 1$, so $\tau=1$. 
From this contradiction we get by Theorem~\ref{thm:haran} that $M$ is Hilbertian.
\end{proof}

Following \cite{FehmPetersen} we call a field extension $E/K$ an \textbf{$\mathcal{H}$-extension} if every intermediate field
$K\subseteq M\subseteq E$ is Hilbertian.

\begin{cor}\label{cor:hextension}
Let $K$ be Hilbertian and $E/K$ a Galois extension of finite abelian-simple length. Then $E/K$ is an $\mathcal{H}$-extension.
\end{cor}

\begin{proof}
Let $K\subseteq M\subseteq E$ be an intermediate field and let $\hat{M}$ be the Galois closure of $M/K$. Then $\Gal(E/K)$ surjects onto $\Gal(\hat{M}/K)$. Thus $M/K$ is of finite abelian-simple length (Lemma \ref{lem:basicpropertiesofmSigma}), and by Theorem~\ref{thm:bdgdlimplieshilb} it follows that $M$ is Hilbertian.
\end{proof}

Theorem \ref{thm:criteria} from the introduction is now an immediate consequence
of Lemma \ref{lem:subnormal} and Corollary \ref{cor:hextension}.

\section{Galois representations}
\label{sec:rep}

\subsection{Compact subgroups of ${\rm GL}_n$}

In this section we summarize a few well-known facts about compact subgroups of ${\rm GL}_n(\overline{\mathbb{Q}}_\ell)$.
For a finite extension $F$ of $\mathbb{Q}_\ell$ we denote by $\mathcal{O}_F$ the integral closure of $\mathbb{Z}_\ell$ in $F$
and by $\mathfrak{m}_F$ the maximal ideal of $\mathcal{O}_F$.

\begin{lem}\label{lem:GL1}
Every compact subgroup of $\GL_n(\overline{\mathbb{Q}}_\ell)$ is contained in $\GL_n(F)$ for a finite extension $F$ of $\mathbb{Q}_\ell$.
\end{lem}

\begin{proof}
See for example \cite[p.~244]{Skinner}. 
\end{proof}

\begin{lem}\label{lem:GL2}
A compact subgroup of $\GL_n(F)$, where $F$ is a finite extension of $\mathbb{Q}_\ell$, is conjugate to a subgroup of 
$\GL_n(\mathcal{O}_F)$.
\end{lem}

\begin{proof}
This can be proven exactly like the special case $F=\mathbb{Q}_\ell$ explained in
\cite[\S 6 Exercise 5a, p.~392]{Bourbaki}.
\end{proof}

\begin{lem}\label{lem:GLnQl}
Every compact subgroup of $\GL_n(\mathcal{O}_F)$, where $F$ is a finite extension of
$\mathbb{Q}_\ell$, is a finitely generated profinite group. 
\end{lem}

\begin{proof}
Let $m=[F:\mathbb{Q}_\ell]$. Then $\mathcal{O}_F$ is a free $\mathbb{Z}_\ell$-module of rank $m$.
The regular representation of $\mathcal{O}_F$ as a $\mathbb{Z}_\ell$-module identifies the given compact subgroup of $\GL_n(\mathcal{O}_F)$
with a closed subgroup of $\GL_{mn}(\mathbb{Z}_\ell)$,
and every such subgroup is finitely generated, see \cite[Proposition 22.14.4]{FriedJarden}.
\end{proof}

\begin{lem}\label{lem:kernel}
If $F$ is a finite extension of $\mathbb{Q}_\ell$,
then the kernel $N$ of the residue map $\GL_n(\mathcal{O}_F)\to \GL_n(\mathcal{O}_F/\mathfrak{m}_F)$ is pro-$\ell$.
\end{lem}

\begin{proof}
For the special case $F=\mathbb{Q}_\ell$, $\ell>2$ see for example \cite[Theorem 5.2]{Dixon}.
For a direct proof of the general case observe that if $\mathfrak{m}_F=\lambda \mathcal{O}_F$, then
$N=\varprojlim_k (I_n + \lambda{\rm Mat}_n(\mathcal{O}_F/\lambda^k))$,
and $I_n + \lambda{\rm Mat}_n(\mathcal{O}_F/\lambda^k)$ is an $\ell$-group.
\end{proof}

\subsection{A consequence of a theorem of Larsen and Pink}
A straightforward application of the following theorem of Larsen-Pink 
allows us to conclude that a compact subgroup of ${\rm GL}_n(\overline{\mathbb{Q}}_\ell)$ is an extension of a profinite group of finite abelian-simple length by a pro-$\ell$ group.

\begin{thm}[Larsen-Pink {\cite{LarsenPink}}]\label{thm:LarsenPink}
For any $n$ there exists a constant  $J(n)$ depending only on $n$ such that any finite  subgroup $\Lambda$ of $\GL_n$ over any field $k$ possesses normal subgroups $\Lambda_3\leq \Lambda_2\leq \Lambda_1$ such that
\begin{enumerate}
\item $[\Lambda:\Lambda_1]\leq J(n)$.
\item Either $\Lambda_1=\Lambda_2$, or $\ell:= {\rm char}(k)$ is positive and $\Lambda_1/\Lambda_2$ is a direct product of finite simple groups of Lie type in characteristic $\ell$.
\item $\Lambda_2/\Lambda_3$ is abelian of order not divisible by ${\rm char}(k)$.
\item Either $\Lambda_3=1$, or $\ell:= {\rm char}(k)$ is positive and $\Lambda_3$ is an $\ell$-group.
\end{enumerate}
\end{thm}

\begin{cor}\label{cor:subgrp}
For any $n$ there exists $m=m(n)$ depending only on $n$ such that every compact subgroup $\Lambda$ of $\GL_n(\mathcal{O}_F)$,
where $F$ is a finite extension of $\mathbb{Q}_\ell$, for some $\ell$, admits a pro-$\ell$ normal subgroup  $N$ such that the abelian-simple length of $\Lambda/N$ is at most $m$.
\end{cor}

\begin{proof}
Let $\Lambda_4\leq \Lambda$ be the intersection of $\Lambda$ with the kernel of the residue map $\GL_n(\mathcal{O}_F)\to \GL_n(\mathcal{O}_F/\mathfrak{m}_F)$. Then $\Lambda/\Lambda_4$ is a subgroup of $\GL_n(\mathcal{O}_F/\mathfrak{m}_F)$ and $\Lambda_4$ is a pro-$\ell$ group by Lemma \ref{lem:kernel}. Note that $\mathcal{O}_F/\mathfrak{m}_F$ is a finite field, so $\Lambda/\Lambda_4$ is a finite group.

Theorem~\ref{thm:LarsenPink} applied to $\Lambda/\Lambda_4$ gives normal subgroups $\Lambda_3 \leq  \Lambda_2\leq \Lambda_1$ of $\Lambda$ that contain $\Lambda_4$ such that $[\Lambda:\Lambda_1]\leq J(n)$, $\Lambda_1/\Lambda_2$ is a product of finite simple groups, $\Lambda_2/\Lambda_3$ abelian, and $\Lambda_3/\Lambda_4$ is an $\ell$-group.

Since $\Lambda_4$ is pro-$\ell$ we get that $N := \Lambda_3$ is also pro-$\ell$.
By Lemma \ref{lem:log_length}, the abelian-simple length of $\Lambda/\Lambda_1$ is at most $\log_2(J(n))$. 
Thus by Proposition \ref{prop:lengthextension} the abelian-simple length of $\Lambda/N$ is bounded by $m(n):=\log_2(J(n)) + 2$.
\end{proof}

\subsection{Proof of Theorem~\ref{thm:main}}
\label{sec:prfGR}

The proof of the following proposition appears in \cite[Proposition 2.4]{FehmPetersen}.

\begin{prop}
Let $K_i$, $i\in I$ be a family of $\mathcal{H}$-extensions of a Hilbertian field $K$ which are Galois over $K$. Assume that there is an $\mathcal{H}$-extension $E/K$ such that the fields $K_i E$, $i\in I$, are linearly disjoint over $E$. Then the compositum $\prod_{i\in I} K_i$ is an $\mathcal{H}$-extension of $K$.\label{lem:obs}
\end{prop}

We now prove Theorem~\ref{thm:main}.
Let $K$ be Hilbertian, let $n\in\mathbb{N}$, let $(\rho_\ell)$ be a family of Galois representations of dimension $n$,
and let $L$ be an algebraic extension of $K$ fixed by $\bigcap_\ell \ker\rho_\ell$.
We want to prove that $L$ is Hilbertian.

Since a purely inseparable extension of a Hilbertian field is Hilbertian, see \cite[Proposition 12.3.3]{FriedJarden},
we can assume without loss of generality that $L/K$ is separable.
Thus, if we denote by $K_\ell$ the fixed field of $\ker\rho_\ell$,
$L$ is contained in the compositum $\prod_{\ell} K_\ell$.

By Lemma \ref{lem:GL1} and Lemma \ref{lem:GL2}, we can assume without loss of generality that for each $\ell$, 
$\im(\rho_\ell)\subseteq\GL_n(\mathcal{O}_{F_\ell})$ for a finite extension $F_\ell$ of $\mathbb{Q}_\ell$.
By Lemma~\ref{lem:GLnQl}, the Galois group $\Gal(K_\ell/K) = \im(\rho_\ell)$ is finitely generated, hence $K_\ell/K$ is an $\mathcal{H}$-extension, cf.~\cite[Lemma 4]{Jarden}.

Applying Corollary~\ref{cor:subgrp} gives a constant $m=m(n)$ (depending only on $n$) and,
for each $\ell$, a Galois extension $N_\ell$ of $K$ that is contained in $K_\ell$ such that $\Gal(K_\ell/N_\ell)$ is pro-$\ell$ and the abelian-simple length of $\Gal(N_\ell/K)$ is at most $m$.
Let $E$ be the compositum of all $N_\ell$. 
By Proposition~\ref{prop:lengthN}, the abelian-simple length of $\Gal(E/K)$ is at most $m$.
Thus $E/K$ is an $\mathcal{H}$-extension by Corollary~\ref{cor:hextension}.

Since $\Gal(K_\ell E/E)$ embeds into $\Gal(K_\ell/N_\ell)$ via the restriction map, it is pro-$\ell$. We thus get that the family $K_\ell E$ is linearly disjoint over $E$. 
By Proposition~\ref{lem:obs}, the compositum $\prod_\ell K_\ell$ is an $\mathcal{H}$-extension of $K$, so $L$ is Hilbertian, as claimed.
\hfill\qed

\section{Further applications}
\label{sec:appl}

\subsection{Finite Galois representations}
\label{sec:appl1}

By a {\bf finite $n$-dimensional representation} of $\Gal(K)$ we mean a continuous homomorphism $\rho:\Gal(K)\rightarrow\GL_n(k)$
with finite image, for some field $k$ (equipped with the discrete topology).
In Theorem~\ref{thm:main}, if instead of taking one $n$-dimensional $\ell$-adic Galois representation for each prime number $\ell$
we take finite $n$-dimensional representations, we can actually handle all such representations simultaneously.

\begin{thm}\label{thm:allrep}
Let $K$ be a Hilbertian field and let $n$ be a fixed integer.
Denote by $\Omega$ the family of all finite $n$-dimensional representations of $\Gal(K)$.
Then every algebraic extension $L$ of $K$ fixed by $\bigcap_{\rho\in\Omega}\ker\rho$ is Hilbertian.
\end{thm}

\begin{proof}
As in the proof of Theorem \ref{thm:main}
we can assume without loss of generality that $L/K$ is separable.
Let $\rho:\Gal(K)\rightarrow\GL_n(k)$ be an element of $\Omega$, let $K_\rho$ be the fixed field of $\ker(\rho)$,
let $\Lambda=\Gal(K_\rho/K)$, and let $\ell={\rm char}(k)$. 
By Theorem \ref{thm:LarsenPink} there exist subgroups $\Lambda_3\leq\Lambda_2\leq\Lambda_1$ of $\Lambda$
such that
$[\Lambda:\Lambda_1]\leq J(n)$, $\Lambda_1/\Lambda_2$ is a product of finite simple groups, $\Lambda_2/\Lambda_3$ abelian, and $\Lambda_3$ 
is an $\ell$-group if $\ell>0$, and trivial otherwise.

Since $\Lambda_3$ is unipotent (every element is annihilated by $X^{\ell^m}-1 = (X-1)^{\ell^m}$ for
some~$m \in \mathbb{N}$), it is conjugate to a subgroup of the group of upper triangular matrices with diagonal $(1,\dots,1)$,
and hence has derived length at most $n-1$, cf.~\cite[p.~87]{Borel}.
This implies that $l(\Lambda_3)\leq n-1$.
Putting everything together we get from Lemma \ref{lem:log_length} and Proposition \ref{prop:lengthextension} that
$l(\Lambda)\leq c:=\log_2(J(n))+1+n$.

Now let $M=\prod_{\rho\in\Omega}K_\rho$ be the compositum.
Then $L\subseteq M$.
Since $l(\Gal(K_\rho/K))\leq c$ for each $\rho\in\Omega$, by
Proposition \ref{prop:lengthN} we get $l(\Gal(M/K))\leq c<\infty$.
Hence, by Corollary \ref{cor:hextension}, $L$ is Hilbertian.
\end{proof}

Recall that if $E/K$ is an elliptic curve then the action of $\Gal(K)$ on the $p$-torsion points $E[p]$ of $E$  induces a Galois representation $\rho_{E,p}\colon \Gal(K)\to \GL_2(\FF_p)$ and the fixed field $K_{\rho_{E,p}}$ of $\ker \rho$ is exactly the field one gets by adjoining the $p$-torsion points of $E$. Hence the following corollary follows from Theorem \ref{thm:allrep}:

\begin{cor}
Let $K$ be a Hilbertian field.
Denote by $K^{\rm tor}$ the field obtained from $K$ by adjoining for each prime number $p$ all the $p$-torsions points $E[p]$
of {\em all} elliptic curves $E/K$. 
Then every subfield of $K^{\rm tor}$ that contains $K$ is Hilbertian.
\end{cor}

Of course, a similar result holds true for all abelian varieties over $K$ of a bounded dimension $d$.

\subsection{Solvable extensions}

A separable algebraic extension $L/K$ is \textbf{solvable} if there exists a Galois extension $N/K$ such that $L\subseteq N$ and $\Gal(N/K)$ is solvable. For example, if $K\subseteq K_1\subseteq\dots\subseteq K_n$ is a tower of abelian extensions, then $K_n/K$ is solvable
(note that the Galois group of the Galois closure of $K_n/K$ has derived length at most $n$).
On the other hand, the maximal pro-solvable extension $\mathbb{Q}^{sol}$ of $\mathbb{Q}$ is not solvable.

Let $K$ be a Hilbertian field and $L/K$ a solvable extension. 
If $L/K$ is Galois, then $L$ can be obtained from $K$ by finitely many abelian steps. 
Thus an immediate application of Kuyk's theorem gives that if $K$ is Hilbertian, then $L$ is Hilbertian. To the best of our knowledge, the same result for non-Galois solvable extensions $L/K$ was out of reach of the previously known Hilbertianity criteria.

\begin{thm}
Let $K$ be a Hilbertian field and $L/K$ a separable algebraic extension. Assume that $L/K$ is solvable. Then $L$ is Hilbertian.
\end{thm}

\begin{proof}
Let $N/K$ be a solvable Galois extension with $L\subseteq N$. Then $N/K$ is of finite derived length, thus of finite abelian-simple length. By Corollary~\ref{cor:hextension}, $L$ is Hilbertian.
\end{proof}

\subsection{Extensions of bounded degree}
Let $K$ be a Hilbertian field, $d$ a fixed integer, and $\{N_i \mid i\in I\}$ a family of finite extensions of $K$ of \emph{degree at most $d$}. 
Let $N=\prod_{i\in I} N_i$ be the compositum. If each $N_i$ is Galois over $K$, then, using Haran's diamond theorem, it is rather easy to deduce that $N$ is Hilbertian. It seems that the same statement in general, i.e.\ when  the $N_i/K$ are not necessarily Galois, is more difficult to achieve, and was unknown.
However, it follows immediately from the following stronger result.

\begin{thm}\label{thm:bd}
Let $K$ be a Hilbertian field, let $d$ be a fixed integer, and let $N$ be the compositum of all extensions of $K$ of degree at most $d$
in some fixed algebraic closure of $K$. Then 
every intermediate field $K\subseteq L \subseteq N$ is Hilbertian.
\end{thm}

\begin{proof}
Since a purely inseparable extension of a Hilbertian field is Hilbertian \cite[Proposition 12.3.3]{FriedJarden}, we can assume without loss of generality that $L$ is contained in the compositum $N_0$ of all separable extensions of $K$ of degree at most $d$. 
Since a separable extension of degree at most $d$ is contained in a Galois extension of degree at most $d!$, 
$N_0$ is contained in the compositum $N_1$ of all Galois extensions of $K$ of degree at most $d!$. Since the abelian-simple length of a Galois extension of degree at most $d!$ is at most $\log_2(d!)$ (Lemma~\ref{lem:log_length}), by Proposition~\ref{prop:lengthN} the abelian-simple length of $\Gal(N_1/K)$ is at most $\log_2(d!)$. Thus by Corollary~\ref{cor:hextension}, $L$ is Hilbertian. 
\end{proof}

\subsection{Rational points on varieties}

Let $K$ be a Hilbertian field and let $V$ be an affine $K$-variety of dimension $n$.
By Theorem \ref{thm:bd} there exists an $\mathcal{H}$-extension $N/K$ such that $V(N)$ is Zariski-dense in $V$.
Indeed, by the Noether normalization lemma there is a finite morphism $f:V\rightarrow\mathbb{A}_K^n$,
and if the degree of $f$ is $d$, then every point in $\mathbb{A}_K^n(K)$ is the image of a point of $V(\bar{K})$ of degree at most $d$,
so if $N$ denotes the compositum of all extensions of $K$ of degree at most $d$, then $\mathbb{A}_K^n(K)\subseteq f(V(N))$
and $V(N)$ is Zariski-dense in $V$.
Actually, if $K$ is countable, one can find one $\mathcal{H}$-extension that works for all $K$-varieties $V$ simultaneously:
Recall that a field $K$ is called {\bf pseudo algebraically closed} if for every absolutely irreducible $K$-variety $V$,
the set of $K$-rational points $V(K)$ is Zariski-dense in $V$,
cf.~\cite[Chapter 11]{FriedJarden}.

\begin{thm}
Every countable Hilbertian field $K$ has a Galois extension $M/K$ such that $M$ is pseudo algebraically closed
and every intermediate field $K\subseteq L\subseteq M$ is Hilbertian.
\end{thm}

\begin{proof}
By \cite[Theorem 18.10.3]{FriedJarden} there exists a Galois extension $M$ of $K$ which is pseudo algebraically closed,
and $\Gal(M/K)\cong\prod_{n=1}^\infty S_n$, where $S_n$ is the symmetric group of degree $n$. Since each $S_n$ has abelian-simple length at most $3$,
so has $\prod_{n=1}^\infty S_n$ by Proposition \ref{prop:lengthN}.
By Corollary \ref{cor:hextension}, every $L$ as in the theorem is Hilbertian.
\end{proof}

A conjecture of Frey and Jarden in \cite{FreyJarden} states that every abelian variety over $\mathbb{Q}$ acquires infinite rank over $\mathbb{Q}^{\rm ab}$, the maximal abelian extension of $\mathbb{Q}$. We cannot prove this but instead show that there is some other $\mathcal{H}$-extension with this property:

\begin{cor}
There exists an algebraic extension $M$ of $\mathbb{Q}$ such that 
every subfield of $M$ is Hilbertian and
every nonzero abelian variety $A/\mathbb{Q}$ acquires infinite rank over $M$.
\end{cor}

\begin{proof}
By \cite{FehmPetersenRank}, every nonzero abelian variety over a pseudo algebraically closed field of characteristic zero has infinite rank.
\end{proof}

\subsection{Free profinite groups}
Using the theory of pseudo algebraically closed fields we get the following freeness criterion for subgroups of the free profinite group of countable rank $\hat{F}_\omega$, cf.~\cite[Chapter 17]{FriedJarden}.

\begin{thm}\label{thm:free}
Let $N\leq M\leq\hat{F}_\omega$ be closed subgroups such that $N$ is normal in $\hat{F}_\omega$ and $\hat{F}_\omega/N$ is of finite abelian-simple length.
Then $M\cong\hat{F}_\omega$.
\end{thm}

\begin{proof}
Recall that a countable pseudo algebraically closed field $K$ is Hilbertian if and only if $\Gal(K)\cong\hat{F}_\omega$, see \cite[Proposition 5.10.1 and Theorem 5.10.3]{Patching}.
Let $K$ be such a countable Hilbertian pseudo algebraically closed field of characteristic 0, cf.~\cite[Example 5.10.7, 2nd paragraph]{Patching}.

Now view $N\leq M$ as subgroups of $\Gal(K)\cong\hat{F}_\omega$ and
let $E$ (respectively $L$) be the fixed field of $N$ (respectively $M$) in an algebraic closure of $K$. 
Then $\Gal(E/K) = \hat{F}_\omega/N$ is of finite abelian-simple length and $L\subseteq E$. 
Thus $L$ is Hilbertian by Corollary~\ref{cor:hextension}, pseudo algebraically closed by \cite[Corollary 11.2.5]{FriedJarden}, and countable. 
Hence, $M=\Gal(L)\cong \hat{F}_\omega$, as claimed. 
\end{proof}

We conclude by pointing out without a complete proof that a more general statement 
holds in the category of pro-$\mathcal{C}$ groups, where $\mathcal{C}$ is an arbitrary Melnikov formation, cf.~\cite[Section 17.3]{FriedJarden}.
The following result can be deduced from Theorem \ref{thm:free} along the lines of \cite[second paragraph of proof of Theorem 3.1]{BarySoroker}.

\begin{cor}
Let $\mathcal{C}$ be a Melnikov formation, let $F=\hat{F}_\omega(\mathcal{C})$ be the free pro-$\mathcal{C}$ group of countable rank and let $N\leq M\leq F$ be closed subgroups. Assume that $N$ is normal in $F$, $F/N$ is of finite abelian-simple length, and $M$ is pro-$\mathcal{C}$. Then $M\cong \hat{F}_\omega(\mathcal{C})$.
\end{cor}

\section*{Acknowledgements}

This work was greatly influenced by the papers \cite{Jarden} of Jarden and \cite{Thornhill} of Thornhill.
The authors would like to thank Sebastian Petersen for interesting discussions on this subject
and Dror Speiser for pointing out the application in Section \ref{sec:appl1}. We are also grateful to the referee for her/his helpful suggestions. 
This research was supported by the Lion Foundation Konstanz and the von Humboldt Foundation. L.~B.~S was partially supported by a grant from the GIF, the German-Israeli Foundation for Scientific Research and Development. 
G.~W.~acknowledges partial support by the Priority Program 1489 of the 
Deutsche Forschungsgemeinschaft.

\bibliographystyle{plain}

 \textsc{School of Mathematical Sciences, 
Tel Aviv University,
Ramat Aviv,
Tel Aviv 69978,
Israel}

 \textit{E-mail address}: \texttt{barylior@post.tau.ac.il}

\textsc{Universit\"at Konstanz, Fachbereich Mathematik und Statistik, Fach D 203, 78457 Konstanz, Germany}

\textit{E-mail address}: \texttt{arno.fehm@uni-konstanz.de }

\textsc{Universit\'e du Luxembourg,
Facult\'e des Sciences, de la Technologie et de la Communication,
6, rue Richard Coudenhove-Kalergi, 
L-1359 Luxembourg,
Luxembourg}

\textit{E-mail address}:  \texttt{gabor.wiese@uni.lu}

\end{document}